\documentclass[12pt,a4paper,reqno]{amsart}

\usepackage{anysize}
\usepackage[latin1]{inputenc}
\usepackage{amsmath}
\usepackage{amssymb}
\usepackage{array}

\marginsize{2.5cm}{2.5cm}{3.5cm}{3.5cm}

\newcommand{\jacobi}[2]{\genfrac{(}{)}{.8pt}{}{#1}{#2}}

\begin{document}

\newtheorem{observacion}{Remark}
\newtheorem{theorem}{Theorem}[section]
\newtheorem{corollary}[theorem]{Corollary}
\newtheorem{lemma}[theorem]{Lemma}
\newtheorem{proposition}[theorem]{Proposition}
\newtheorem{definition}[theorem]{Definition}
\newtheorem{example}[theorem]{Example}
\newtheorem{remark}[theorem]{Remark}
\newtheorem{notation}[theorem]{Notation}
\newtheorem{question}[theorem]{Question}
\newtheorem{claim}[theorem]{Claim}
\newtheorem{conjecture}[theorem]{Conjecture}

\title{Ramanujan-like series for $1/\pi^2$ and String Theory}

\author{Gert Almkvist}
\address{Institute of Algebraic Meditation, Fogdar\"{o}d 208 S-24333 H\"{o}\"{o}r, SWEDEN}
\email{gert.almkvist@yahoo.se}

\author{Jesús Guillera}
\address{Av.\ Ces\'areo Alierta, 31 esc.~izda 4$^\circ$--A, Zaragoza, SPAIN}
\email{jguillera@gmail.com}

\dedicatory{Dedicated to Herbert Wilf on his $80^{th}$ birthday}

\date{}

\maketitle

\begin{abstract}
Using the machinery from the theory of Calabi-Yau differential equations, we find formulas for $1/\pi^2$ of hypergeometric and non-hypergeometric types.
\end{abstract}

\section{Introduction.}

Almost 100 years ago, Ramanujan found 17 formulas for $1/\pi$. The most spectacular was
\[
\sum_{n=1}^{\infty }\frac{\left(\frac12\right)_{n}\left(\frac14\right)_{n}\left(\frac34\right)_{n}}{n!^{3}}(26390n+1103) \frac{1}{99^{4n+2}}=\frac{\sqrt{2}}{4 \pi},
\]
where $(a)_{0}=1$ and $(a)_{n}=a(a+1)\cdots(a+n-1)$ for $n>1$ is the Pochhammer symbol. The formulas were not proved until in the 1980:ies by the Borwein brothers using modular forms (see \cite{Bo} and the recent surveys \cite{BaBeCh} and \cite{Zu2}).

In 2002 the second author found seven similar formulas for $1/\pi^2$. Three of them were proved using the WZ-method (see \cite{Gu1},\cite{Gu3}, \cite{Gu4}). Others, like
\[
\sum_{n=1}^{\infty} \frac{\left(\frac12 \right)_n \left(\frac18\right)_{n}\left(\frac38\right)_n \left(\frac58\right)_n\left(\frac78\right)_{n}}{n!^5}(1920n^{2}+304n+15) \frac{1}{7^{4n}} = \frac{56\sqrt{7}}{\pi^2}
\]
(see \cite{Gu2}, \cite{Gu4}), were found using PSLQ to find the triple $(1920,304,15)$ after guessing $z=7^{-4}$. This was inspired by the similar formula for $1/\pi$
\[
\sum_{n=1}^{\infty }\frac{\left(\frac12\right)_{n}\left(\frac14\right)_{n}\left(\frac34\right)_{n}}{n!^{3}}(40n+3)\frac{1}{7^{4n}} =\frac{49\sqrt{3}}{9\pi}.
\]

To avoid guessing $z$, the second author, using $ 5\times 5-$matrices, developed a technique to find $z$, while instead guessing a rather small rational number $k.$ To that purpose one had to solve an equation of type
\[
\frac{1}{6}\log ^{3}(q)-\nu_{1}\log (q)-\nu_{2}-T(q)=0,
\]
where $\nu_{1}$ depends on $k$ linearly, $\nu_{2}$ is a constant, and $T(q)$ is a certain power series (see \cite{Gu5}).
It was suggested by Wadim Zudilin that $T(q)$ had to do with the Yukawa coupling $K(q)$ of the fourth order pullback of the fifth order differential equation satisfied by the sum for general $z.$ The exact relation is
\[
(q\frac{d}{dq})^{3}T(q)=1-K(q).
\]
This is explained and proved here. The reason that it works is that all differential equations involved are Calabi-Yau. The theory results in a simplified and very fast Maple program to find $z$. As a result we mention the new formula
\begin{equation}\label{new-ramalike}
\frac{1}{\pi^{2}}=32 \sum_{n=0}^{\infty}\frac{(6n)!}{3 \cdot n!^{6}}(532n^{2}+126n+9)\frac{1}{10^{6n+3}},
\end{equation}
where the summands contain no infinite decimal fractions. However this is not a BBP-type (Bailey-Borwein-Plouffe) series \cite{BaBoMaWi} and, due to the factorials, it is not useful to extract individual decimal digits of $1/\pi^2$. (The look we have written the formula above is courtesy by Pigulla).

\section{Calabi-Yau differential equations.}

\subsection{Formal definitions.}

A Calabi-Yau differential equation is a $4^{th}$ order differential equation with rational coefficients
\begin{equation*}
y^{(4)}+c_{3}(z)y^{\prime \prime \prime}+c_{2}(z)y^{\prime \prime}+c_{1}(z)y^{\prime}+c_{0}(z)y=0
\end{equation*}
satisfying the following conditions.
\par \textbf{1.} It is MUM (Maximal Unipotent Monodromy), i.e. the indicial equation at $z=0$ has zero as a root of order $4$. It means that there is a Frobenius solution of the following form%
\begin{equation*}
y_{0}=1+A_{1}z+A_{2}z^{2}+\cdots,
\end{equation*}
\begin{equation*}
y_{1}=y_{0}\log (z)+B_{1}z+B_{2}z^{2}+\cdots,
\end{equation*}
\begin{equation*}
y_{2}=\frac{1}{2}y_{0}\log ^{2}(z)+(B_{1}z+B_{2}z^{2}+\cdots)\log (z)+C_{1}z+C_{2}z^{2}+\cdots,
\end{equation*}
\begin{equation*}
y_{3}=\frac{1}{6}y_{0}\log ^{3}(z)+\frac{1}{2}(B_{1}z+B_{2}z^{2}+\cdots)\log
^{2}(z)+(C_{1}z+C_{2}z^{2}+\cdots)\log (z)+D_{1}z+D_{2}z^{2}+\cdots.
\end{equation*}
It is very useful that Maple's "formal\_sol" produces the four solutions in exactly this form (though labeled $1-4$)
\par \textbf{2.} The coefficients of the equation satisfy the identity
\begin{equation*}
c_{1}=\frac{1}{2}c_{2}c_{3}-\frac{1}{8}c_{3}^{3}+c_{2}^{\prime }-\frac{3}{4}%
c_{3}c_{3}^{\prime }-\frac{1}{2}c_{3}^{\prime \prime}.
\end{equation*}
\par\textbf{3.} Let $t=$ $y_{1}/y_{0}.$ Then
\begin{equation*}
q=\exp (t)=z+e_{2}z^{2}+\cdots
\end{equation*}%
can be solved
\begin{equation*}
z=z(q)=q-e_{2}q^{2}+\cdots,
\end{equation*}%
which is called the "mirror map". We also construct the "Yukawa coupling" defined by
\begin{equation*}
K(q)=\frac{d^{2}}{dt^{2}}(\frac{y_{2}}{y_{0}}).
\end{equation*}
This can be expanded in a Lambert series
\begin{equation*}
K(q)=1+\sum_{d=1}^{\infty }n_{d}\frac{d^{3}q^{d}}{1-q^{d}},
\end{equation*}%
where the $n_{d}$ are called "instanton numbers". For small $d$ the $n_{d}$ are conjectured to count rational curves of degree $d$ on the corresponding Calabi-Yau manifold. Then the third condition is
\par \textbf{(a)} $y_{0}$ has integer coefficients.
\par \textbf{(b)} $q$ has integer coefficients.
\par \textbf{(c)} There is a fixed integer $N_{0}$ such that all $ N_{0}n_{d} $ are integers. \\
In \cite{Al1} the first author shows how to discover Calabi-Yau differential equations.

\subsection{Pullbacks of $5^{th}$ order equations.}

The condition \textbf{2} is equivalent to
\begin{equation*}
\mathbf{2'}: \qquad
\begin{vmatrix}
y_{0} & y_{3} \\
y_{0}^{\prime} & y_{3}^{\prime}
\end{vmatrix}
=
\begin{vmatrix}
y_{1} & y_{2} \\
y_{1}^{\prime} & y_{2}^{\prime}
\end{vmatrix}.
\end{equation*}
This means that the six wronskians formed by the four solutions to our Calabi-Yau equation reduce to five. Hence they satisfy a $5^{th}$ order differential equation
\begin{equation*}
w^{(5)}+d_{4}w^{(4)}+d_{3}w^{\prime \prime \prime }+d_{2}w^{\prime \prime}+d_{1}w^{\prime }+d_{0}w=0.
\end{equation*}
The condition $\textbf{2}$ for the $4^{th}$ order equation leads to a corresponding condition for the $5^{th}$ order equation
\begin{equation*}
\mathbf{2_5}: \qquad d_{2}=\frac{3}{5}d_{3}d_{4}-\frac{4}{25}d_{4}^{3}+\frac{3}{2}d_{3}^{\prime }-
\frac{6}{5}d_{4}d_{4}^{\prime }-d_{4}^{\prime \prime}.
\end{equation*}

Conversely given a fifth order equation satisfying \textbf{2}$_{\text{\textbf{5}}}$ with solution $w_{0}$ we can find a pullback, i.e. a fourth order equation with solutions $y_{0},y_{1},\dots$ such that $w_{0}=z(y_{0}y_{1}^{\prime }-y_{0}^{\prime }y_{1})$ . There is another pullback, $\widehat{y}$ which often cuts the degree into half. It was discovered by Yifan Yang and is simply a multiple $\widehat{y}=gy$ of the ordinary pullback where
\begin{equation*}
g=z^{-1/2}\exp (\frac{3}{10}\int d_{4}dz).
\end{equation*}
In the proof below, all formulas contain only quotients of solutions so the factor $g$ cancels, so it is irrelevant if we use ordinary or YY-pullbacks. Since the $q(z)$ are the same, so are the inverse functions $z(q).$

\subsection{The proof.}
Consider
\[
w_{0}(z)=\sum_{n=0}^{\infty }\frac{(1/2)_{n}(s_1)_{n}(1-s_1)_{n}(s_2)_{n}(1-s_2)_{n}}{n!^{5}}(\rho z)^{n}=
\sum_{n=0}^{\infty }A_{n}z^{n},
\]
which satisfies the differential equation
\[
\left\{\theta ^{5}-\rho z(\theta +\frac{1}{2})(\theta +s_1)(\theta +1-s_1)(\theta+s_2)(\theta +1-s_2)\right\} w_{0}=0,
\]
where $\theta =z\frac{d}{dz}$. The equation satisfies $\mathbf{2'}$, so
\[
w_{0}=z(y_{0}y_{1}^{\prime }-y_{0}^{\prime }y_{1})
\]
where $y_{0}$ and $y_{1}$ satisfy a fourth order differential equation (the ordinary pullback). We will consider the following 14 cases (compare the 14 hypergeometric Calabi-Yau equations in the "Big Table" (see \cite{AlEnStZu}).

\begin{center}
{\setlength{\extrarowheight}{0.75ex}
\vskip 0.5cm \textbf{Table 1. Hypergeometric cases} \vskip 0.25cm
\begin{tabular}{|l||l|l|l|l|}
\hline
\# & $s_1$ & $s_2$ & $\rho$ & $A_{n}$ \\[1.1ex] \hline \hline
$\widetilde{1}$ & $1/5$ & $2/5$ & $4\cdot 5^{5}$ & $\binom{2n}{n}^{3}\binom{3n}{n}\binom{5n}{2n}$ \\[0.75ex] \hline
$\widetilde{2}$ & $1/10$ & $3/10$ & $4\cdot 8\cdot 10^{5}$ & $\binom{2n}{n}^{2}\binom{3n}{n}\binom{5n}{2n}\binom{10n}{5n}$ \\[0.75ex] \hline
$\widetilde{3}$ & $1/2$ & $1/2$ & $4\cdot 2^{8}$ & $\binom{2n}{n}^{5}$ \\[0.75ex]  \hline
$\widetilde{4}$ & $1/3$ & $1/3$ & $4\cdot 3^{6}$ & $\binom{2n}{n}^{3}\binom{3n}{n}^{2}$ \\[0.75ex]  \hline
$\widetilde{5}$ & $1/2$ & $1/3$ & $4\cdot 2^{4}\cdot 3^{3}$ & $\binom{2n}{n}^{4}\binom{3n}{n}$ \\[0.75ex]  \hline
$\widetilde{6}$ & $1/2$ & $1/4$ & $4\cdot 2^{10}$ & $\binom{2n}{n}^{4}\binom{4n}{2n}$ \\[1.1ex]  \hline
$\widetilde{7}$ & $1/8$ & $3/8$ & $4\cdot 2^{16}$ & $\binom{2n}{n}^{3}\binom{4n}{2n}\binom{8n}{4n}$ \\[0.75ex]  \hline
$\widetilde{8}$ & $1/6$ & $1/3$ & $4\cdot 2^{4}\cdot 3^{6}$ & $\binom{2n}{n}^{3}\binom{4n}{2n}\binom{6n}{2n}$ \\[0.75ex]  \hline
$\widetilde{9}$ & $1/12$ & $5/12$ & $4\cdot 12^{6}$ & $\binom{2n}{n}^{3}\binom{6n}{2n}\binom{12n}{6n}$ \\[0.75ex]  \hline
$\widetilde{10}$ & $1/4$ & $1/4$ & $4\cdot 2^{12}$ & $\binom{2n}{n}^{3}\binom{4n}{2n}^{2}$ \\[0.75ex]  \hline
$\widetilde{11}$ & $1/4$ & $1/3$ & $4\cdot 12^{3}$ & $\binom{2n}{n}^{3}\binom{3n}{n}\binom{4n}{2n}$ \\[0.75ex]  \hline
$\widetilde{12}$ & $1/6$ & $1/4$ & $4\cdot 2^{10}\cdot 3^{3}$ & $\binom{2n}{n}^{2}\binom{3n}{n}\binom{4n}{2n}\binom{6n}{3n}$ \\[0.75ex]  \hline
$\widetilde{13}$ & $1/6$ & $1/6$ & $4\cdot 2^{8}\cdot 3^{6}$ & $\binom{2n}{n}\binom{3n}{n}^{2}\binom{6n}{3n}^{2}$ \\[0.75ex]  \hline
$\widetilde{14}$ & $1/2$ & $1/6$ & $4\cdot 2^{8}\cdot 3^{3}$ & $\binom{2n}{n}^{3}\binom{3n}{n}\binom{6n}{3n}$ \\[0.75ex]  \hline
\end{tabular}
} \vskip 0.5cm
\end{center}

Assume that the formula
\[
\sum_{n=0}^{\infty }A_{n}(a+bn+cn^{2})z^{n}=\frac{1}{\pi^{2}},
\]
is a Ramanujan-like one; that is the numbers $a$, $b$, $c$ and $z$ are algebraic. Then, in \cite{Gu5} it is conjectured that we have an expansion
\begin{equation}\label{expansion}
\sum_{n=0}^{\infty}A_{n+x}(a+b(n+x)+c(n+x)^{2})z^{n+x}=\frac{1}{\pi^{2}}-\frac{k}{2}x^{2}+\frac{j}{24}\pi^{2}x^{4}
+O(x^{5}),
\end{equation}
where $k$ and $j$ are rational numbers. It holds in all known examples (in fact $3k$ and $j$ are integers). However there is a better argument to support the conjecture. It consists in comparing with the cases $_3F_2$ of Ramanujan-type series for $1/\pi$, for which the second author proved in \cite{Gu5} that $k$ must be rational.

In $A_x$ we replace $x!$ by $\Gamma(x+1)$ (Maple does it automatically). Later we use the harmonic number $H_n=1+1/2+\cdots+1/n$ which is replaced by $H_x=\psi(x+1)-\gamma$, where $\psi(x)=\Gamma'(x)/\Gamma(x)$ and $\gamma$ is Euler's constant.

\par Expansion (\ref{expansion}) can be reformulated in the way
\begin{equation}\label{re-expansion}
\sum_{n=0}^{\infty }\frac{A_{n+x}}{A_x}(a+b(n+x)+c(n+x)^{2})z^{n}=\frac{1}{z^{x}A_{x}}
( \frac{1}{\pi ^{2}}-\frac{k}{2}x^{2}+\frac{j}{24}\pi ^{2}x^{4}+\cdots).
\end{equation}
Write
\[
\sum_{n=0}^{\infty }\frac{A_{n+x}}{A_x}z^{n}=\sum_{i=0}^{\infty }a_{i}x^{i}, \quad
\sum_{n=0}^{\infty }\frac{A_{n+x}}{A_x}(n+x)z^{n}=\sum_{i=0}^{\infty }b_{i}x^{i}
\]
and
\[
\sum_{n=0}^{\infty }\frac{A_{n+x}}{A_x}(n+x)^{2}z^{n}=\sum_{i=0}^{\infty }c_{i}x^{i},
\]
where $a_{i}$, $b_{i}$, $c_{i}$ are power series in $z$ with rational coefficients. They are related to
the solutions $w_{0},w_{1},w_{2},w_{3},w_{4}$ of the fifth order differential equation
\begin{align}
& w_{0}=a_{0} \nonumber \\
& w_{1}=a_{0}\log (z)+a_{1} \nonumber \\
& w_{2}=a_{0}\dfrac{\log ^{2}(z)}{2}+a_{1}\log (z)+a_{2} \nonumber \\
& w_{3}=a_{0}\dfrac{\log ^{3}(z)}{6}+a_{1}\dfrac{\log ^{2}(z)}{2}+a_{2}\log (z)+a_{3} \nonumber \\
& w_{4}=a_{0}\dfrac{\log ^{4}(z)}{24}+a_{1}\dfrac{\log ^{3}(z)}{6}+a_{2} \dfrac{\log ^{2}(z)}{2}+a_{3}\log (z)+a_{4}. \nonumber \end{align}
We also have $b_{0}=za_{0}^{\prime}$ and $b_{k}=a_{k-1}+za_{k}^{\prime}$ for $k=1,2,3,4.$
\par If we write the expansion of $A_x$ in the form
\begin{equation}\label{Ax-expan}
A_x=1+\frac{e}{2} \pi^2 x^2-h \zeta(3) x^3+ (\frac{3e^2}{8} - \frac{f}{2}) \pi^4 x^4 + O(x^5),
\end{equation}
then for the right hand side $M$ of (\ref{re-expansion}) we have
\[
M=\frac{1}{z^{x}A_{x}}(\frac{1}{\pi ^{2}}-\frac{k}{2}x^{2}+\frac{j}{24}\pi
^{2}x^{4}+\cdots)=m_{0}+m_{1}x+m_{2}x^{2}+m_{3}x^{3}+m_{4}x^{4}+\cdots,
\]
where
\begin{align}
& m_{0}=\dfrac{1}{\pi ^{2}}, \nonumber \\
& m_{1}=-\dfrac{1}{\pi ^{2}}\log (z), \nonumber \\
& m_{2}=\dfrac{1}{\pi ^{2}}\left\{ \dfrac{1}{2}\log ^{2}(z)-\dfrac{\pi ^{2}}{2}(k+e)\right\}, \nonumber \\
& m_{3}=\dfrac{1}{\pi ^{2}}\left\{ -\dfrac{1}{6}\log ^{3}(z)+\dfrac{\pi ^{2}}{2}(k+e)\log (z)+h\zeta (3)\right\}. \nonumber
\end{align}
and
\[ 2m_0m_4-2m_1m_3+m_2^2=\frac{j}{12}+\frac{k^2}{4}+ek+f. \]
Here
\[ e=\dfrac{5}{3}+\cot^{2}(\pi s_1)+\cot^{2}(\pi s_2), \qquad f=\frac{1}{\sin^{2}(\pi s_1) \sin^{2}(\pi s_2)} \]
and
\[ h=\dfrac{2}{\zeta (3)}\left\{ \zeta (3,1/2)+\zeta (3,s_1)+\zeta (3,1-s_1)+\zeta(3,s_2)+\zeta (3,1-s_2)\right\}, \] where
\[ \zeta (s,a)=\sum_{n=0}^{\infty }\frac{1}{(n+a)^{s}} \]
is Hurwitz $\zeta -$function. If one uses $A_{n}$ defined by binomial coefficients Maple finds the values of $e$ and $h$ directly. We conjecture that the 14 pairs $(s_1,s_2)$ given in the table are the only rational $(s_1,s_2)$ between $0$ and $1$ making $h$ an integer. Note that the same $(s_1,s_2)$ give the only hypergeometric Calabi-Yau differential equations (see \cite{Al2} and \cite{Al3}).

\begin{center}
\vskip 0.5 cm \textbf{Table 2. Values of e, h, f} \vskip 0.25cm
{\setlength{\extrarowheight}{2.0ex}
\begin{tabular}{|c||c|c|c|c|c|c|c|c|c|c|c|c|c|c|}
\hline
\# & $\widetilde{1}$  &  $\widetilde{2}$  &  $\widetilde{3}$  & $\widetilde{4}$  & $\widetilde{5}$  &  $\widetilde{6}$  &  $\widetilde{7}$  & $\widetilde{8}$  & $\widetilde{9}$  & $\widetilde{10}$  & $\widetilde{11}$  & $\widetilde{12}$  &  $\widetilde{13}$    &  $\widetilde{14}$    \\[2.0ex] \hline \hline
e  &  $\dfrac{11}{3}$   &  $\dfrac{35}{3}$   &  $\dfrac{5}{3}$   &  $\dfrac{7}{3}$   &  $2$   &  $\dfrac{8}{3}$   &  $\dfrac{23}{3}$   &  $5$   &  $\dfrac{47}{3}$   &   $\dfrac{11}{3}$   &      $3$ &  $\dfrac{17}{3}$    &   $\dfrac{23}{3}$   &  $\dfrac{14}{3}$    \\[2.0ex] \hline
h  &  $42$   &  $290$   &  $10$   &  $18$   &  $14$   &   $24$  &  $150$   &  $70$   &  $486$   &   $38$   &  $28$    &   $80$   &   $122$   &  $66$    \\[2.0ex] \hline
f  &  $\dfrac{16}{5}$  &  $16$   &  $1$   &  $\dfrac{16}{9}$   &  $\dfrac{4}{3}$   &  $2$   &   $8$  &  $\dfrac{16}{3}$   &   $16$  &   $4$   &   $\dfrac{8}{3}$   &  $8$    &   $16$   &   $4$   \\[2.0ex] \hline
\end{tabular}
} \vskip 0.5cm
\end{center}

Now we want to use many of the identities for the wronskians in \cite[pp.4-5]{Al2}. Therefore we invert the formulas
\begin{align}
& a_{0}=w_{0}, \nonumber \\
& a_{1}=w_{1}-w_{0}\log (z), \nonumber \\
& a_{2}=w_{2}-w_{1}\log (z)+w_{0}\dfrac{\log ^{2}(z)}{2}, \nonumber \\
& a_{3}=w_{3}-w_{2}\log (z)+w_{1}\dfrac{\log ^{2}(z)}{2}-w_{0}\dfrac{\log ^{3}(z)}{6}, \nonumber \\
& a_{4}=w_{4}-w_{3}\log (z)+w_{2}\dfrac{\log ^{2}(z)}{2}-w_{1}\dfrac{\log ^{3}(z)}{6}+w_{0}\dfrac{\log ^{4}(z)}{24} \nonumber \end{align}
and
\begin{align}
& b_{0}=zw_{0}^{\prime}, \nonumber \\
& b_{1}=z(w_{1}^{\prime}-w_{0}^{\prime}\log (z)), \nonumber \\
& b_{2}=z(w_{2}^{\prime}-w_{1}^{\prime}\log (z)+w_{0}^{\prime }\dfrac{\log^{2}(z)}{2}), \nonumber \\
& b_{3}=z(w_{3}^{\prime}-w_{2}^{\prime}\log (z)+w_{1}^{\prime }\dfrac{\log^{2}(z)}{2}-w_{0}^{\prime}\dfrac{\log ^{3}(z)}{6}), \nonumber \\
& b_{4}=z(w_{4}^{\prime}-w_{3}^{\prime}\log (z)+w_{2}^{\prime }\dfrac{\log^{2}(z)}{2}-w_{1}^{\prime}\dfrac{\log ^{3}(z)}{6}+w_{0}^{\prime}\dfrac{\log^{4}(z)}{24}). \nonumber
\end{align}
The key equation in \cite{Gu5} is
\begin{equation}\label{mH}
m_{3}=H_{0}m_{0}-H_{1}m_{1}+H_{2}m_{2},
\end{equation}
where
\[
H_{0}=\frac{a_{0}b_{4}-a_{4}b_{0}}{a_{0}b_{1}-a_{1}b_{0}}, \qquad
H_{1}=\frac{a_{0}b_{3}-a_{3}b_{0}}{a_{0}b_{1}-a_{1}b_{0}}, \qquad
H_{2}=\frac{a_{0}b_{2}-a_{2}b_{0}}{a_{0}b_{1}-a_{1}b_{0}}.
\]
We get ($g$ is a multiplicative factor defined in \cite[p.5]{Al2}. It will cancel out)
\[
a_{0}b_{1}-a_{1}b_{0}=z
\left|
\begin{tabular}{ll}
$w_{0}$ & $w_{1}$ \\
$w_{0}^{\prime}$ & $w_{1}^{\prime}$
\end{tabular}
\right|
=z^3 g \, y_0^2.
\]
("The double wronskian is almost the square")
\[
a_{0}b_{2}-a_{2}b_{0}=z\left|
\begin{tabular}{ll}
$w_{0}$ & $w_{2}$ \\
$w_{0}^{\prime }$ & $w_{2}^{\prime }$
\end{tabular}
\right|-z\log (z) \left|
\begin{tabular}{ll}
$w_{0}$ & $w_{1}$ \\
$w_{0}^{\prime }$ & $w_{1}^{\prime }$
\end{tabular}
\right|=z^{3} g \left\{ y_{0}y_{1}-y_{0}^{2}\log (z)\right\}.
\]
It follows
\[
H_{2}=\frac{z^{3} g \left\{ y_{0}y_{1}-y_{0}^{2}\log (z)\right\} }{z^{3} g \, y_{0}^{2}}=\frac{y_{1}}{y_{0}}
-\log (z)=\log (q)-\log (z)=\log (\frac{q}{z}).
\]
Furthermore
\begin{align}
a_{0}b_{3}-a_{3}b_{0} &=z \left|
\begin{tabular}{ll}
$w_{0}$ & $w_{3}$ \\
$w_{0}^{\prime }$ & $w_{3}^{\prime}$
\end{tabular}
\right|-z\log(z) \left|
\begin{tabular}{ll}
$w_{0}$ & $w_{2}$ \\
$w_{0}^{\prime }$ & $w_{2}^{\prime}$
\end{tabular} \right|
+z\frac{\log^{2}(z)}{2} \left|
\begin{tabular}{ll}
$w_{0}$ & $w_{1}$ \\
$w_{0}^{\prime }$ & $w_{1}^{\prime}$
\end{tabular}
\right|
\nonumber \\
&=z^{3} g \left\{ \frac{1}{2}y_{1}^{2}-y_{0}y_{1}\log (z)+y_{0}^{2}\frac{\log^{2}(z)}{2}\right\} \nonumber
\end{align}
and
\[
H_{1}=\frac{1}{2}(\frac{y_{1}}{y_{0}})^{2}-\frac{y_{1}}{y_{0}}\log (z)+
\frac{\log ^{2}(z)}{2}=\frac{1}{2}\log^2(\frac{q}{z}).
\]
Finally we have that
\begin{align}
a_{0}b_{4}-a_{4}b_{0} &=z\left|
\begin{tabular}{ll}
$w_{0}$ & $w_{4}$ \\
$w_{0}^{\prime }$ & $w_{4}^{\prime}$
\end{tabular}
\right| -z\log (z) \left|
\begin{tabular}{ll}
$w_{0}$ & $w_{3}$ \\
$w_{0}^{\prime }$ & $w_{3}^{\prime}$
\end{tabular} \nonumber
\right|
\\
&+z\frac{\log ^{2}(z)}{2} \left|
\begin{tabular}{ll}
$w_{0}$ & $w_{2}$ \\
$w_{0}^{\prime}$ & $w_{2}^{\prime}$
\end{tabular}
\right|-z\frac{\log ^{3}(z)}{6} \left|
\begin{tabular}{ll}
$w_{0}$ & $w_{1}$ \\
$w_{0}^{\prime}$ & $w_{1}^{\prime}$
\end{tabular} \right| \nonumber
\\
&=z^{3} g \left\{ \frac{1}{2}(y_{1}y_{2}-y_{0}y_{3})-\frac{1}{2}y_{1}^{2}\log
(z)+y_{0}y_{1}\frac{\log ^{2}(z)}{2}-y_{0}^{2}\frac{\log ^{3}(z)}{6}\right\} \nonumber
\end{align}
and
\[
H_{0}=\frac{1}{2}(\frac{y_{1}}{y_{0}}\frac{y_{2}}{y_{0}}-\frac{y_{3}}{y_{0}})-\frac{1}{2}t^{2}\log (z)+
t\frac{\log ^{2}(z)}{2}-\frac{\log ^{3}(z)}{6}.
\]

Substituting these formulas into (\ref{mH}), we obtain
\begin{multline} \nonumber
\dfrac{1}{\pi ^{2}}\left\{ -\dfrac{1}{6}\log ^{3}(z)+\dfrac{\pi ^{2}}{2}(k+e)\log (z)+h\zeta (3)\right\}
\\=\frac{1}{\pi ^{2}}\left\{ \frac{1}{2}(\frac{y_{1}}{y_{0}}\frac{y_{2}}{y_{0}}-\frac{y_{3}}{y_{0}})-\frac{1}{2}t^{2}\log (z)+t\frac{\log ^{2}(z)}{2}-\frac{\log ^{3}(z)}{6}\right\}
\\+\frac{1}{\pi ^{2}}\log (z)\left\{ \frac{t^{2}}{2}-t\log (z)+\frac{\log^{2}(z)}{2}\right\} +\frac{1}{\pi ^{2}}(t-\log (z))(\frac{\log ^{2}(z)}{2}-\frac{\pi ^{2}}{2}(k+e)),
\end{multline}
which simplifies to
\[
\frac{1}{2}(\frac{y_{1}}{y_{0}}\frac{y_{2}}{y_{0}}-\frac{y_{3}}{y_{0}})-
\dfrac{\pi ^{2}}{2}(k+e)\log (q)-h\zeta (3)=0.
\]
Here
\[
\Phi =\frac{1}{2}(\frac{y_{1}}{y_{0}}\frac{y_{2}}{y_{0}}-\frac{y_{3}}{y_{0}})
\]
is wellknown in String Theory and is called the Gromov-Witten potential (up to a multiplicative constant, (see \cite[p.28]{CoKa}). It is connected to the Yukawa coupling $K(q)$ by
\[ (q\frac{d}{dq})^{3}\Phi =K(q). \]
Writing $\Phi =\frac{1}{6}\log^{3}(q)-T(q)$ (see lemma 2.1) we get the following equation for finding $q$ and hence $z$ for given $k$
\begin{equation}\label{Tq}
\frac{1}{6}t^3-\dfrac{\pi^{2}}{2}(k+e)t-h\zeta (3)-T(q)=0, \qquad q=\exp(t).
\end{equation}
We look for real algebraic solutions of $z$. To look for alternating series, that is if $z<0$ all we need to do is replacing $q=\exp(t)$ with $q=-\exp(t)$  in (\ref{Tq}). In order to make a quick sieve of the solutions, once we get $q$ we compute $j$ and see if it is an integer (or rational with small denominator). Using the formulas \cite[eqs. 3.48 \& 3.50]{Gu5} we find
\begin{equation}\label{Uq}
j=12\left\{ \frac{1}{\pi^{4}}(\frac{1}{2}t^{2}-q\frac{d}{dq}T(q)-
\frac{\pi ^{2}}{2}(k+e))^{2}-\frac{k^{2}}{4}-ek-f\right\}.
\end{equation}
\begin{lemma}
The function $T(q)$ is a power series with $T(0)=0$.
\end{lemma}
\begin{proof}
We have
\[ y_{1}=y_{0}\log (z)+\alpha _{1} \]
which implies
\[ \frac{y_{1}}{y_{0}}=\log (q)=\log (z)+\frac{\alpha _{1}}{y_{0}}=\log (z)+\beta _{1} \]
and hence
\[
\log (z)=\log (q)-\beta _{1},
\]
where $\alpha _{1}$ and $\beta _{1}=\dfrac{\alpha _{1}}{y_{0}}$ are power series without constant term. Furthermore
\[ y_{2}=y_{0}\frac{\log ^{2}(z)}{2}+\alpha _{1}\log (z)+\alpha _{2} \]
leads to
\[
\frac{y_{2}}{y_{0}}=\frac{1}{2}(\log (q)-\beta _{1})^{2}+\beta _{1}(\log (q)-\beta _{1})+\beta _{2}
=\frac{1}{2}\log ^{2}(q)+\beta _{2}-\frac{1}{2}\beta _{1}^{2},
\]
where $\beta _{2}=\dfrac{\alpha _{2}}{y_{0}}$ with $\beta _{2}(0)=0$. Finally
\[ y_{3}=y_{0}\frac{\log ^{3}(z)}{6}+\alpha _{1}\frac{\log ^{2}(z)}{2}+\alpha_{2}\log (z)+\alpha _{3} \]
and
\[ \frac{y_{3}}{y_{0}}=\frac{1}{6}(\log (q)-\beta _{1})^{3}+\frac{1}{2}\beta _{1}(\log (q)-\beta _{1})^{2}+\beta _{2}(\log (q)-\beta _{1})+\beta _{3}, \]
where $\beta _{3}=\dfrac{\alpha _{3}}{y_{0}}$ with $\beta _{3}(0)=0$.
Collecting terms we have
\[
\frac{1}{2}(\frac{y_{1}}{y_{0}}\frac{y_{2}}{y_{0}}-\frac{y_{3}}{y_{0}})= \frac{1}{6}\log ^{3}(q)-
\frac{1}{2}(\beta _{3}-\beta _{1}\beta _{2}+\frac{1}{3}\beta _{1}^{3}),
\]
which proves the lemma.
\end{proof}

\section{Computations}

\subsection{Hypergeometric differential equations.}

In only half of the $14$ cases have we found solutions to the equation (\ref{Tq}), where the indicator $j$ is an integer. Using \cite[eqs. 3.47-3.48]{Gu5}, we have the following formula for computing $c$
\begin{equation}\label{hyper-cz}
\tau=\frac{c}{\sqrt{1-\rho z}},
\end{equation}
where
\[ \tau^{2}=\frac{j}{12}+\frac{k^{2}}{4}+ek+f. \]
Then $a$ and $b$ can be computed by \cite[eq. 3.45]{Gu5} or by PSLQ. Here are our results where the series converges

\begin{center}
{\setlength{\extrarowheight}{1.5ex}
\vskip 0.5cm \textbf{Table 3. Convergent hypergeometric Ramanujan-like series for $1/\pi^2$} \vskip 0.25cm
\begin{tabular}{|l||c|c|c|c|c|c|c|}
\hline \hline
\# & $k$ & $j$ & $z_{0}$ & $\tau ^{2}$ & $a$ & $b$ & $c$ \\[1.5ex] \hline \hline
$\widetilde{3}$ & $1$ & $25$ & $-\dfrac{1}{2^{12}}$ & $5$ & $\dfrac{1}{8}$ & $1$ & $\dfrac{5}{2}$ \\[1.5ex] \hline
$\widetilde{3}$ & $5$ & $305$ & $-\dfrac{1}{2^{20}}$ & $41$ & $\dfrac{13}{128}$ & $\dfrac{45}{32}$ & $\dfrac{205}{32}$ \\[1.5ex] \hline
$\widetilde{5}$ & $\dfrac{2}{3}$ & $16$ & $\dfrac{1}{2^{12}}$ & $\dfrac{37}{9}$ & $\dfrac{1}{16}$ & $\dfrac{9}{16}$ & $\dfrac{37}{24}$ \\[1.5ex] \hline
$\widetilde{5}$ & $\dfrac{8}{3}$ & $112$ & $(\frac{5\sqrt{5}-11}{8})^{3}$ & $\dfrac{160}{9}$ & $56-25\sqrt{5}$ & $303-135\sqrt{5}$ & $\dfrac{1220}{3}-180\sqrt{5}$ \\[1.5ex] \hline
$\widetilde{6}$ & $2$ & $80$ & $\dfrac{1}{2^{16}}$ & $15$ & $\dfrac{3}{32}$ & $\dfrac{17}{16}$ & $\dfrac{15}{4}$ \\[1.5ex] \hline
$\widetilde{7}$ & $8$ & $992$ & $\dfrac{1}{2^{18}7^{4}}$ & $168$ & $\dfrac{15}{392}\sqrt{7}$ & $\dfrac{38}{49}\sqrt{7}$ & $\dfrac{240}{49}\sqrt{7}$ \\[1.5ex] \hline
$\widetilde{8}$ & $\dfrac{5}{3}$ & $85$ & $-\dfrac{1}{2^{18}}$ & $\dfrac{193}{9}$ & $\dfrac{15}{128}$ & $\dfrac{183}{128}$ & $\dfrac{965}{192}$ \\[1.5ex] \hline
$\widetilde{8}$ & $15$ & $2661$ & $-\dfrac{1}{2^{18}3^{6}5^{3}}$ & $\dfrac{1075}{3}$ & $\dfrac{29}{640}\sqrt{5}$ & $\dfrac{693}{640}\sqrt{5}$ & $\dfrac{2709}{320}\sqrt{5}$ \\[1.5ex] \hline
$\widetilde{8}$ & $\dfrac{8}{3}$ & $160$ & $\dfrac{1}{2^{6}5^{6}}$ & $\dfrac{304}{9}$ & $\dfrac{36}{375}$ & $\dfrac{504}{375}$ & $\dfrac{2128}{375}$ \\[2.0ex] \hline
$\widetilde{11}$ & $3$ & $157$ & $-\dfrac{1}{2^{12}3^{4}}$ & $27$ & $\dfrac{5}{48}$ & $\dfrac{21}{16}$ & $\dfrac{21}{4}$ \\[1.5ex] \hline
$\widetilde{12}$ & $7$ & $757$ & $-\dfrac{1}{2^{22}3^{3}}$ & $123$ & $\dfrac{15}{768}\sqrt{3}$ & $\dfrac{278}{768}\sqrt{3}$ & $\dfrac{205}{96}\sqrt{3}$ \\[1.5ex] \hline
\end{tabular}
} \vskip 0.5cm
\end{center}
In all the hypergeometric cases there a singular solution when $k=j=0$ (it has not a corresponding Ramanujan-like series). For that solution we have $z=1/\rho$, $a=b=c=0$.
\par In addition we have found the solutions $\widetilde{3}$: $k=0$, $j=3$, $z=-2^{-8}$, $a=1/4$, $b=3/2$, $c=5/2$ and $\widetilde{11}$: $k=1/3$, $j=13$, $z=-2^{-12}$, $a=3/16$, $b=25/16$, $c=43/12$, for which the corresponding series are "divergent" \cite{GuiZu}.
\par Although our new program, which evaluates the function $T(q)$ much faster, has allowed us to try all rational values of $k$ of the form $k=i/60$ with $0 \leq i \leq 1200$ the only new series that we have found is for $\widetilde{8}$ with $k=8/3$, and it is
\[
\sum_{n=0}^{\infty }\frac{(6n)!}{n!^{6}}(532n^{2}+126n+9)\frac{1}{10^{6n}}=\frac{375}{4\pi^{2}},
\]
that is (\ref{new-ramalike}). A brief story of the discovery of the other $10$ formulas in the table is in \cite{Gu7}.
\par Finally we give a hypergeometric example of different nature in case $\widetilde{3}$. Take $z_0=-2^{-10}$, $q_0=q(z_0)$, $t_0=\log|q_0|$ and $T(q)$ of $\widetilde{3}$, we find using PSLQ, among the quantities $T(q_0)$, $t_0^3$, $t_0^2 \, \pi$, $t_0 \, \pi^2$, $\pi^3$, $\zeta(3)$ the following remarkable relation:

\[ \frac{1}{6}(t_0+\pi)^3-\frac{5}{6}\pi^2(t_0+\pi)- \frac{\pi^3}{3}-10\zeta(3)-T(q_0)=0. \]
The theory we have developed allows to understand that the last relation has to do with the following formula proved by Ramanujan \cite[p.41]{Be}:
\[ \sum_{n=0}^{\infty} \frac{(-1)^n}{2^{10n}} \binom{2n}{n}^5 (4n+1)=\frac{2}{\Gamma^4(\frac{3}{4})}. \]
To see why we guess that
\[
\frac{\Gamma^4(\frac{3}{4})}{2} \sum_{n=0}^{\infty} \frac{(-1)^n}{2^{10(n+x)}} \binom{2n+2x}{n+x}^5 [4(n+x)+1]=
 1-\pi x + \frac{\pi^2}{2} x^2 + \frac{\pi^3}{6} x^3 - \frac{19\pi^4}{24} x^4 + O(x^5)
\]
by expanding the first side numerically. Hence
\[
2^{10x} \, \binom{2x}{x}^{-5} \frac{2}{\Gamma^4(\frac{3}{4})}(1-\pi x + \frac{\pi^2}{2} x^2 +
\frac{\pi^3}{6} x^3)=m_0+m_1x+m_2x^2+m_3x^3+O(x^4)
\]
and we get $m_0$, $m_1$, $m_2$, $m_3$. Finally we use identity (\ref{mH}) replacing $\log(z)$ with $\log 2^{-10}$.

\subsection{Non-hypergeometric differential equations.}

If we write the ordinary pullback in the form
\[
\theta_z^4 y = \left[ e_3(z) \theta_z^3  + e_2(z) \theta_z^2  + e_1(z) \theta_z  +e_0(z) \right] y, \qquad \theta_z=z \frac{d}{dz},
\]
then, the generalization of the relation (\ref{hyper-cz}) is
\begin{equation}\label{nonhyper-tau-c}
\tau = c \left( \exp \int \frac{e_3(z)}{2z} dz \right), \qquad \tau^2=\frac{j}{12}+\frac{k^{2}}{4}+ek+f.
\end{equation}
We say that a solution is singular if it does not has a corresponding Ramanujan-like series. We conjecture that $h$ is the unique rational number such that singular solutions exists. The numbers $e$ and $f$ are not so important because they can be absorbed in $k$ and $j$ respectively. However, to agree with the hypergeometric cases we will choose $e$ and $f$ in such a way that a singular solution takes place at $k=j=0$. This fact allows us to determine the values of the numbers $e$, $h$ and $f$ from (\ref{Tq}) and (\ref{Uq}) using the PSLQ algorithm. For many sequences $A(n)$ there exists a finite value of $z$ which is singular, then we can get this value solving the equation
\[ \frac{d z(q)}{dq}=0. \]
In the sequel, we will show in several tables the rational values of the invariants $e$, $h$ and $f$ followed by the series found.

\begin{center}
\vskip 0.5cm \textbf{Table $\alpha$}. \vskip 0.25 cm
{\setlength{\extrarowheight}{1.75ex}
\begin{tabular}{|c||c|c|c|c|}
\hline
& $ \#39=A \ast \alpha $ & $ \#61=B\ast \alpha $ & $\#37=C\ast \alpha $ & $\#66=D\ast \alpha $ \\[1.75ex] \hline \hline
$e, h, f$ & $1, \dfrac{14}{3}, \dfrac{1}{3} $ & $\dfrac{4}{3}, \dfrac{26}{3}, \dfrac{4}{9}$ & $2, \dfrac{56}{3}, \dfrac{2}{3}$ & $4, \dfrac{182}{3}, \dfrac{4}{3}$  \\[2.0ex] \hline
\end{tabular}
} \vskip 0.5cm
\end{center}
For $A \ast \alpha$, taking $k=1/3$ we get $j=5$, and we discover the series
\[
\sum_{n=0}^{\infty} \binom{2n}{n}^2 \sum_{i=0}^{n} \binom{n}{i}^2\binom{2i}{i}\binom{2n-2i}{n-i}\frac{(-1)^n}{2^{8n}} (40n^2+26n+5)=\frac{24}{\pi^2}.
\]
This series was first conjectured by Zhi-Wei Sun \cite{Sun} inspired by p-adic congruences.

\begin{center}
\vskip 0.5cm \textbf{Table $\epsilon$}. \vskip 0.25cm
{\setlength{\extrarowheight}{1.75ex}
\begin{tabular}{|c||c|c|c|c|}
\hline
& $ \#122=A \ast \epsilon $ & $ \#170=B\ast \epsilon $ & $C\ast \epsilon $ & $D\ast \epsilon $ \\[1.75ex] \hline \hline
$e, h, f$ & $\dfrac{7}{6}, \dfrac{45}{8}, \dfrac{1}{2}$ & $\dfrac{3}{2}, \dfrac{77}{8}, \dfrac{2}{3}$ & $\dfrac{13}{6}, \dfrac{157}{8}, 1$ & $\dfrac{25}{6}, \dfrac{493}{8}, 2$  \\[2.0ex] \hline
\end{tabular}
} \vskip 0.5cm
\end{center}
For $B \ast \epsilon$, taking $k=1$ we get $j=22$, and we find the formula
\[
\sum_{n=0}^{\infty} \binom{2n}{n} \binom{3n}{n} \sum_{i=0}^n \binom{n}{i}^2\binom{2i}{n}^2 \frac{1}{2^{7n}3^{3n}}(1071n^2+399n+46)=\frac{576}{\pi^2}.
\]

\begin{center}
\vskip 0.5cm \textbf{Table $\beta$}. \vskip 0.25cm
{\setlength{\extrarowheight}{1.75ex}
\begin{tabular}{|c||c|c|c|c|}
\hline
& $ \#40=A \ast \beta $ & $ \#49=B\ast \beta $ & $\#43=C\ast \beta $ & $\#67=D\ast \beta $ \\[1.75ex] \hline \hline
$e,h,f$ & $\dfrac{2}{3}, 3, \dfrac{1}{4}$ & $1, 7, \dfrac{1}{4}$ & $\dfrac{5}{3}, 17, \dfrac{1}{4}$ & $\dfrac{11}{3}, 59, \dfrac{1}{4}$  \\[2.0ex] \hline
\end{tabular}
} \vskip 0.5cm
\end{center}
For $A \ast \beta$, taking $k=1$ we get $j=13$, and we have the series
\[
\sum_{n=0}^{\infty} \binom{2n}{n}^2 \sum_{i=0}^{n} \binom{2i}{i}^2 \binom{2n-2i}{n-i}^2  \frac{1}{2^{10n}} (36n^2+12n+1)=\frac{32}{\pi^2}.
\]
For $B \ast \beta$, taking $k=1/3$ we get $j=1$, and we find
\[
\sum_{n=0}^{\infty} \frac{(3n)!}{n!^3} \sum_{i=0}^{n} \binom{2i}{i}^2 \binom{2n-2i}{n-i}^2  \frac{1}{2^{9n}} (25n^2-15n-6)=\frac{192}{\pi^2}.
\]
\newpage
\begin{center}
\vskip 0.5cm \textbf{Table $\delta$}. \vskip 0.25cm
{\setlength{\extrarowheight}{1.75ex}
\begin{tabular}{|c||c|c|c|c|}
\hline
& $ A \ast \delta $ & $ B\ast \delta $ & $C\ast \delta $ & $D\ast \delta $ \\[1.75ex] \hline \hline
$e, h, f$ & $1, \dfrac{9}{2}, \dfrac{17}{36}$ & $\dfrac{4}{3}, \dfrac{17}{2}, \dfrac{7}{12}$ & $2, \dfrac{37}{2}, \dfrac{29}{36}$ & $4, \dfrac{121}{2}, \dfrac{53}{36}$  \\[2.0ex] \hline
\end{tabular}
} \vskip 0.5cm
\end{center}
For $A \ast \delta$, taking $k=2/3$ we get $j=28/3$, and we have
\[
\sum_{n=0}^{\infty} \binom{2n}{n}^2 \sum_{i=0}^n \frac{(-1)^i 3^{n-3i} (3i)!}{i!^3} \binom{n}{3i} \binom{n+i}{i} \frac{(-1)^n}{3^{6n}} (803n^2+416n+68)=\frac{486}{\pi^2}.
\]

\begin{center}
\vskip 0.5cm \textbf{Table $\theta$}. \vskip 0.25cm
{\setlength{\extrarowheight}{1.75ex}
\begin{tabular}{|c||c|c|c|c|}
\hline
& $ A \ast \theta $ & $ B\ast \theta $ & $C\ast \theta $ & $D\ast \theta $ \\[1.75ex] \hline \hline
$e, h, f$ & $\dfrac{2}{3}, -4, 1$ & $1, 0, 1$ & $\dfrac{5}{3}, 10, 1$ & $\dfrac{11}{3}, 52, 1$  \\[2.0ex] \hline
\end{tabular}
} \vskip 0.5cm
\end{center}
For $A \ast \theta$, taking $k=2$ we get $j=56$, and we discover the series
\[
\sum_{n=0}^{\infty} \binom{2n}{n}^2 \sum_{i=0}^{n} 16^{n-i} \binom{2i}{i}^3 \binom{2n-2i}{n-i} \frac{(-1)^n}{2^{13n}} (18n^2+7n+1)=\frac{4 \sqrt{2}}{\pi^2}.
\]
This series was first discovered by Zhi-Wei Sun \cite{Sun} inspired by p-adic congruences.

\par \noindent
For $B \ast \theta$ we get $T(q)=0$ and from the equations we see that for every rational $k$ the value of $j$ is rational as well. Hence for every rational value of $k$ we get a Ramanujan-like series for $1/\pi^2$. For example, for $k=160/3$, we have

\[
\sum_{n=0}^{\infty} \frac{3n!}{n!^3} \sum_{i=0}^{n} 16^{n-i} \binom{2i}{i}^3 \binom{2n-2i}{n-i} P(n) \frac{(-1)^n}{640320^{3n}}=\frac{(2^{4} \cdot 3 \cdot 5 \cdot 23\cdot 29)^{3}}{\pi ^{2}},
\]
where
\[
P(n)=22288332473153467n^{2}+16670750677895547n+415634396862086,
\]
which is the "square" \cite{Zu4} of the brothers Chudnovsky's formula \cite{BaBeCh}
\[
\sum_{n=0}^{\infty }(-1)^{n}\frac{(6n)!}{(3n)!n!^3} (545140134n+13591409)\frac{1}{640320^{3n}}=\frac{53360\sqrt{640320}}{\pi}.
\]
For $C \ast \theta$, taking $k=1$ we get $j=25$ and taking $k=5$ we get $=305$, and we recover the two series proved by W. Zudilin in \cite{Zu1} by doing a quadratic transformation of case $\widetilde{3}$.

\par \noindent In \cite{Al5} the first author, by transforming known formulas given by the second author, found formulas for $1/\pi^2$ where the coefficients belong to the Calabi-Yau equations $\widehat{3}$, $\widehat{5}$, $\widehat{6}$, $\widehat{7}$, $\widehat{8}$, $\widehat{11}$, $\widehat{12}$. Here we list some new ones for the cases $\widehat{3}, \widehat{5}, \widehat{8},\widehat{11}$ and $\# 77$, some found by solving equation (\ref{Tq}).

\subsubsection*{Transformation $\widehat{5}$}
Here
\begin{equation*}
A_{n}=\sum_{i=0}^{n}(-1)^{i} 1728^{n-i}\dbinom{n}{i}\dbinom{2i}{i}^{4}\dbinom{3i}{i}
\end{equation*}
Using $e=2$, $h=14$, $f=4/3$ we find for $k=8/3$ that $j=112$ and $z_{0}=-[320(131+61\sqrt{5})]^{-1}.$ To find the coefficients we had to use the formulas in \cite{Al5}. The resulting formula is
\begin{multline}
\sum_{n=0}^{\infty} A_{n}
\left(    (28765285482\sqrt{5}-64321133730)+ (10068363-4502709\sqrt{5})n \right. \\ \left. +(54\sqrt{5}-122)n^2     \right)\frac{(-1)^n}{(320(131+61\sqrt{5}))^{n}} = \frac{300(1170059408\sqrt{5}-24977012149)}{\pi^{2}}. \nonumber
\end{multline}
By the PSLQ algorithm, trying products of powers of $2$ and $7$ in the denominator of $z$, we see that
\begin{equation*}
\sum_{n=0}^{\infty }A_{n}(n^{2}-63n+300)\frac{1}{1792^{n}}=\frac{4704}{\pi^{2}},
\end{equation*}
but we cannot find the pair $(k,j)$ with our program because the convergence in this case is too slow.

\subsubsection*{Transformation $\widehat{\mathbf{8}}$}
Here
\begin{equation*}
A_{n}=\sum_{i=0}^{n}(-1)^{k}6^{6n-6i}\dbinom{n}{i}\dbinom{2i}{i}^{3}\dbinom{4i}{2i}\dbinom{6i}{2i}.
\end{equation*}%
Using $e=5$, $h=70$, $f=16/3$ we find for $k=5/3$ that $j$ $=85$ and $z_{0}=308800^{-1}$. This allows us to get the formula
\begin{equation*}
\sum_{n=0}^{\infty }A_{n}(16777216n^{2}-3336192n-2912283)\frac{1}{308800^{n}}=
\frac{3\cdot 5^{5} \cdot 193^{2}}{5^{5}\pi^{2}}.
\end{equation*}%
For $k=\dfrac{8}{3}$ we get $j=160$ and the formula
\begin{equation*}
\sum_{n=0}^{\infty }(-1)^{n}A_{n}(48828125n^{2}+17859375n+3649554)\frac{1}{953344^{n}}=
\frac{2^{8}\cdot 3\cdot 7^{5}\cdot 19^{2}}{5^{4}\pi^{2}}.
\end{equation*}

\subsubsection*{Transformation $\widehat{\mathbf{11}}$}
Here
\begin{equation*}
A_{n}=\sum_{i=0}^{n}(-1)^{i}6912^{n-i}\dbinom{n}{i}\dbinom{2i}{i}^{3}\dbinom{3i}{i}\dbinom{4i}{2i}.
\end{equation*}
Using $e=3$, $h=28$, $f=8/3$ we find for $k=1/3$ that $j=13$ and we get the formula
\begin{equation*}
\sum_{n=0}^{\infty }A_{n}(512n^{2}-1992n-225)\frac{1}{11008^{n}}=\frac{3\cdot 43^{2}}{2\pi^{2}}.
\end{equation*}

\subsubsection*{Transformation $\widehat{\mathbf{3}}$}
Here
\begin{equation*}
A_{n}=\sum_{i=0}^{n}(-1)^{i}1024^{n-i}\dbinom{n}{i}\dbinom{2i}{i}^{5}.
\end{equation*}
Transforming two divergent series in \cite{GuiZu} with $z_{0}=-2^{-8}$ and $z_{0}=-1$ respectively (the second one given only implicitly), we obtain two (slowly) convergent formulas
\begin{equation*}
\sum_{n=0}^{\infty }A_{n}(2n^{2}-18n+5)\frac{1}{1280^{n}}=\frac{100}{\pi^{2}}
\end{equation*}
and
\begin{equation*}
\sum_{n=0}^{\infty }A_{n}(n^{2}-2272n+392352)\frac{1}{1025^{n}}=\frac{16\cdot 5253125}{\pi^{2}}.
\end{equation*}
This last identity converges so slowly that the power of our computers seems not enough to check it numerically.

\subsubsection*{Transformation $\# 77$}

Here
\begin{equation*}
A_{n}=\dbinom{2n}{n}\sum_{i=0}^{n}\dbinom{n}{i}\dbinom{2i}{i}^{3}\dbinom{4i}{2i}.
\end{equation*}
The pullback is equivalent to $\widetilde{6}$ (i.e. has the same $K(q)$), so we try the same parameters; $e=8/3$, $h=24$, $f=2$. For $k=2$ we get $j=80$ and $z_{0}=1/65540$. To find $a,b,c$ we have to find the transformation between $\widetilde{6}$ and \#77. Indeed
\begin{equation*}
\sum_{n=0}^{\infty} \dbinom{2n}{n}^{4}\dbinom{4n}{2n}z^{n}=\frac{1}{\sqrt{1-4z}}\sum_{n=0}^{\infty
}A_{n}(\frac{z}{1+4z})^{n}
\end{equation*}
(the sequence of numbers $\widetilde{6}$ is in the left side) and using the method in \cite{Al5}, (see also \cite{AlStZu}), we obtain
\begin{equation*}
\sum_{n=0}^{\infty }A_{n}(402653184n^{2}+114042880n+10051789)\frac{1}{65540^{n}}=\frac{5^{2}\cdot 29^{3}\cdot 113^{3}}{2^{6}\pi ^{2}\sqrt{16385}}.
\end{equation*}

\section{Supercongruences}

Zudilin \cite{Zu3} observed that the hypergeometric formulas for $1/\pi^2$ lead to supercongruences of the form
\[
\sum_{n=0}^{p-1}A_n(a+bn+cn^2)z^n \equiv a \jacobi{d}{p} p^2 \pmod{p^5},
\]
where the notation $(d \ \! | \! \ p)$ stands for the Legendre symbol. Our computations show that for our new Ramanujan-like series for $1/\pi^2$ (\ref{new-ramalike}), we have again a supercongruence following Zudilin's pattern, namely

\begin{equation*}
\sum_{n=0}^{p-1}\dbinom{2n}{n}^{3}\dbinom{4n}{2n}\dbinom{6n}{2n}(532n^{2}+126n+9)\frac{1}{1000000^{n}}\equiv 9p^{2} \pmod{p^5},
\end{equation*}
valid for primes $p\geq 7$.

For superconguences for $\widetilde{5}$ and $k=8/3$, which involves algebraic numbers, see \cite{Gu6}.
For the non-hypergeometric formulas the best one can hope for is a congruence $\pmod{p^3}.$ We give some new ones which agree with Zudilin's observations for \cite[eq. 35]{Zu3}.

\subsubsection*{With Hadamard product $ \#170=B \ast \epsilon$}

\[
\sum_{n=0}^{p-1} \binom{2n}{n} \binom{3n}{n} \sum_{i=0}^n \binom{n}{i}^2\binom{2i}{n}^2 \frac{1}{2^{7n}3^{3n}}(1071n^2+399n+46) \equiv 46 p^2 \pmod{p^3},
\]
for primes $p \geq 5$.

\subsubsection*{With Hadamard product $ \#49 \ B\ast \beta$}
\begin{equation*}
\sum_{n=0}^{p-1}\sum_{i=0}^{n}\dbinom{2n}{n}\dbinom{3n}{n}\dbinom{2i}{i}^{2} \dbinom{2n-2i}{n-i}^{2}(25n^{2}-15n-6)\frac{1}{512^{n}}\equiv -6p^2 \pmod{p^3},
\end{equation*}
for primes $p \geq 7$.

\subsubsection*{With Hadamard product $A \ast \delta$}
\[
\sum_{n=0}^{p-1} \binom{2n}{n}^2 \sum_{i=0}^n \frac{(-1)^i 3^{n-3i} (3i)!}{i!^3} \binom{n}{3i} \binom{n+i}{i} \frac{(-1)^n}{3^{6n}} (803n^2+416n+68) \equiv 68 p^2 \pmod{p^3},
\]
for primes $p \geq 5$.

\subsubsection*{With Hadamard product $C \ast \theta $}

\begin{equation*}
\sum_{n=0}^{p-1}\sum_{i=0}^{n}16^{n-i} \dbinom{2n}{n}\dbinom{4n}{2n}\dbinom{2i}{i}^{3}\dbinom{2n-2i}{n-i}
\frac{18n^{2}-10n-3}{80^{2n}} \equiv -3 \jacobi{5}{p} p^{2}\pmod{p^3},
\end{equation*}
for primes $p \geq 5$ and
\begin{multline*}
\sum_{n=0}^{p-1}\sum_{i=0}^{n}16^{n-i}\dbinom{2n}{n}\dbinom{4n}{2n}\dbinom{2i}{i}^{3}\dbinom{2n-2i}{n-i}
\frac{1046529n^{2}+227104n+16032}{1050625^{n}} \\ \equiv 16032 \jacobi{41}{p} p^{2}\pmod{p^3},
\end{multline*}
for primes $p \geq 7$ and $p\neq 41$.

\section{Conclusion}

We have recovered the $10$ hypergeometric Ramanujan series in \cite{Gu7} and found a new one that the second author missed. But more important, finding the relation among the function $T(q)$ and the Gromov-Witten potential has allowed us to generalize the conjectures of the second author in \cite{Gu5} to the case of non-hypergeometric Ramanujan-Sato like series. Then, by getting $e$, $h$ and $f$ from a singular solution (it always exists) we have solved our equations finding several nice non-hypergeometric series for $1/\pi^2$. Finally, we have checked the corresponding supercongruences of Zudilin-type.

\section*{Appendix 1. A Maple program for case $\widetilde{\mathbf{8}}$.}

We use the YY-pullback found in \cite{Al2}. In order to also treat the case when $z$ and $q$ are negative we introduce a sign $u=\pm 1$.

\begin{verbatim}

with(combinat):
p(1):=expand(-36*(2592*n^4+5184*n^3+6066*n^2+3474*n+755)):
p(2):=expand(2^4*3^10*(4*n+3)*(4*n+5)*(12*n+11)*(12*n+13)):
V:=proc(n) local j:
if n=0 then 1; else sum(stirling2(n,j)*z^j*Dz^j,j=1..n); fi; end:
L:=collect(V(4)+add(add(z^m*coeff(p(m),n,k)*V(k),m=1..2),k=0..4),Dz):
Order:=51:
with(DEtools):
r:=formal_sol(L,[Dz,z],z=0):
y0:=r[4]: y1:=r[3]: y2:=r[2]: y3:=r[1]:
q:=series(exp(y1/y0),z=0,51): m:=solve(series(q,z)=s,z):
convert(simplify(series(subs(z=m,1/2*(y1*y2/y0^2-y3/y0)),s=0,51)),
        polynom):
T:=coeff(-%,ln(s),0):
e:=5: h:=70: f:=16/3:

H:=proc(u) local k,y,z0,q0,j,y0,yy,i,jj; y0:=-10;
for i from 0 to 60 do
k:=i/3;
Digits:=50;
yy:=fsolve(y^3/6-Pi^2/2*(k+e)*y-h*Zeta(3)-subs(s=u*exp(y),T),y=y0);
q0:=exp(%); y0:=yy;
z0:=evalf(subs(s=u*q0,convert(m,polynom)));
j:=evalf(12*(1/Pi^4*(1/2*log(q0)^2-subs(s=u*q0,s*diff(T,s))
         -Pi^2/2*(k+e))^2-k^2/4-e*k-f));
jj:=convert(j,fraction,12);
if denom(jj)<30 then print([k,1/z0,j]); fi; od; end:

\end{verbatim}

Copy and paste the program in Maple and execute $H(1)$ and $H(-1)$. You will get the following results:

\begin{align}
&\left[ \frac{8}{3},        1.00000000000000000000000000000000000000000000000000  \, 10^6,     \right. \nonumber \\
&\qquad \left. \frac{}{}    160.000000000000000000000000000000000000000000000007               \right] \nonumber \\
&\left[ \frac{5}{3},        -2.62143999999999999999999999999999999999999999999996 \, 10^5,     \right. \nonumber \\
&\qquad \left. \frac{}{}    85.0000000000000000000000000000000000000000000000000               \right] \nonumber \\
&\left[ 15,                 -2.3887871999999999999999999999999999999999999999995  \, 10^{10}, \frac{}{} \right. \nonumber \\
&\qquad \left. \frac{}{}    2660.99999999999999999999999999999999999999999999996               \right] \nonumber
\end{align}

To use the program with other cases one has to change the values of $e$, $h$, $f$ and replace the polynomials $p(1)$, $p(2)$, etc, with those corresponding to the new pullback and the number $2$ in $m=1..2$, with the total number of polynomials.

\subsection*{Acknowledgement.}

We would like to thank Wadim Zudilin who, even in his exile on the other side of the earth, has shown great interest in our work.

\end{document}